\documentclass[11pt]{article}
\usepackage{amsmath, amsthm, amssymb, geometry, enumitem, graphicx, booktabs}
\usepackage{hyperref}
\usepackage{mathtools}
\usepackage{authblk}
\usepackage{orcidlink}

\title{Path--Averaged Contractions: A New Generalization of the Banach Contraction Principle\footnote{MSC2020: 47H10, 54E50 \\Keywords: Path Averaged Contractions, contractive mapping, fixed point, complete metric space.}}
\author{Nicola Fabiano\, \orcidlink{0000-0003-1645-2071}}
\affil{``Vin\v{c}a'' Institute of Nuclear Sciences - National 
Institute of the Republic of Serbia, University of Belgrade, Mike Petrovi\'{c}a 
Alasa 12--14, 11351 Belgrade, Serbia; nicola.fabiano@gmail.com}
\date{}

\newtheorem{theorem}{Theorem}

\newtheorem{proposition}[theorem]{Proposition}
\newtheorem{definition}[theorem]{Definition}
\newtheorem{example}[theorem]{Example}
\newtheorem{remark}[theorem]{Remark}

\begin{document}

\maketitle


\begin{abstract}
We introduce a novel class of self-mappings on metric spaces, called {PA-contractions} (Path-Averaged Contractions), defined by an averaging condition over iterated distances. Unlike classical pointwise conditions, the PA-condition allows for transient expansion provided the long-term orbital average contracts. We prove that every continuous PA-contraction on a complete metric space has a unique fixed point, and the Picard iterates converge to it. This condition strictly generalizes that of the classical Banach contraction principle, accommodating mappings where pointwise contraction fails locally. We provide examples showing that PA-contractions are not contained in the classes of F-contractions, Kannan, Chatterjea, or \'Ciri\'c contractions. A comparison table highlights the distinctions. The PA-condition offers a robust framework for analyzing asymptotic behavior in iterative processes where local instability occurs.

\end{abstract}

\maketitle

\section{Introduction}
The Banach Contraction Principle \cite{Banach1922} is a cornerstone of nonlinear analysis. It states that if $(X, \mathrm{d})$ is a complete metric space and $T\colon X \to X$ satisfies
\begin{equation}
\mathrm{d}(Tx, Ty) \leq a \mathrm{d}(x,y) \quad \text{for some } a \in (0,1),
\label{banach}
\end{equation}
then $T$ has a unique fixed point. Numerous generalizations have been proposed to relax the pointwise condition \eqref{banach}, including Kannan \cite{Kannan1968}, \'Ciri\'c \cite{Ciric1971,Ciricorbit,Ciric1974,Ciric1979}, and Chatterjea~\cite{Chatterjea1972}. A systematic comparison of early variants was provided by Rhoades \cite{Rhoades1977}. More recently, {F-contractions} introduced by Wardowski \cite{Wardowski2012} extended the principle using auxiliary functions. They were extended and refined by several
authors~\cite{fabiano2022,Hussain2012,Piri2014,Secelean2013}.

However, existing generalizations largely remain \textit{pointwise}: they require the contraction inequality to hold for every pair $(x,y)$ at every step. In practical applications-such as numerical algorithms, control theory, or stochastic processes-iterates may experience \textit{transient expansion} due to noise, discretization errors, or local nonlinearities, even if the long-term behavior is contractive. A pointwise condition fails to capture such mappings, potentially excluding valid convergent processes.

In this paper, we propose a fundamentally different generalization: instead of pointwise inequalities, we impose a \textit{mean contraction condition} over the orbit of pairs. This leads to the notion of {PA-contraction} (Path-Averaged Contraction), which depends on the cumulative behavior of iterates. This approach is inspired by ergodic theory and path integral formulations \cite{fabianopathintegral}, where global averages often stabilize local fluctuations.

{Motivation and Advantages:}
\begin{enumerate}
    \item Transient Expansion: PA-contractions allow $\mathrm{d}(Tx, Ty) > \mathrm{d}(x,y)$ for specific pairs or early iterates, provided the average over $n$ steps contracts. This models systems with initial instability.
    \item Robustness: Averaging smooths out local anomalies, making the condition potentially more robust in numerical implementations where exact pointwise bounds are hard to verify.
    \item Strict Generalization: We demonstrate mappings that are PA-contractions but fail all classical pointwise conditions (Banach, Kannan, etc.).
\end{enumerate}

We prove a fixed point theorem for PA-contractions, show they generalize Banach contractions, and provide counterexamples demonstrating that PA-contractions are not contained in other classical classes. Finally, we include a comparison table of major contraction types. Important results for PA-contractions in complete
b-metric spaces
have already been obtained and published~\cite{Fpathaveraged2,Fpathaveraged3}.
For general background on fixed point theory, see \cite{Agarwal2011}.

\section{PA-Contractions}
We begin with the central definition.
\begin{definition}[PA-Contraction]
Let $(X,\mathrm{d})$ be a metric space. A mapping $T\colon X \to X$ is called a {PA-contraction} (Path-Averaged Contraction) if there exists $\alpha \in (0,1)$ and $N \in \mathbb{N}$ such that for all $x, y \in X$ and all $n \ge N$,
\begin{equation}
\frac{1}{n} \sum_{k=0}^{n-1} \mathrm{d}(T^{k+1}x, T^{k+1}y) \leq \alpha \cdot \frac{1}{n} \sum_{k=0}^{n-1} \mathrm{d}(T^k x, T^k y).
\label{pa1}
\end{equation}
Equivalently, in sum form:
\begin{equation}
\sum_{k=0}^{n-1} \mathrm{d}(T^{k+1}x, T^{k+1}y) \leq \alpha \sum_{k=0}^{n-1} \mathrm{d}(T^k x, T^k y).
\label{pa2}
\end{equation}
\end{definition}
The condition requires that the average distance between successive iterates contracts by a uniform factor $\alpha < 1$, regardless of the starting pair $(x,y)$, provided the averaging window $n$ is sufficiently large ($n \ge N$).

\begin{remark}
If $N = 1$, then setting $n = 1$ in \eqref{pa1} yields $\mathrm{d}(Tx, Ty) \leq \alpha \mathrm{d}(x, y)$, so $T$ is a Banach contraction. However, for $N > 1$, the condition allows transient expansion, and the mapping need not satisfy any pointwise contraction. The class of PA-contractions thus {strictly generalizes} Banach contractions, as shown by Examples~\ref{ex:panotbanach} and~\ref{ex:continuous}.
\end{remark}

\section{Fixed Point Theorem}
Our main result is the following.
\begin{theorem}[Fixed Point Theorem for PA-Contractions]
Let $(X,\mathrm{d})$ be a complete metric space, and let $T\colon X \to X$ be a continuous PA-contraction. Then $T$ has a unique fixed point $x^* \in X$, and for any $x_0 \in X$, the Picard sequence $x_n = T^n x_0$ converges to $x^*$.
\end{theorem}
\begin{proof}
Let $x_0 \in X$ be arbitrary. Define $x_n = T^n x_0$. Let $$a_k = \mathrm{d}(x_k, x_{k+1}) = \mathrm{d}(T^k x_0, T^{k+1} x_0).$$
Apply the PA-condition~\eqref{pa2} to $x = x_0$, $y = T x_0$. Then  $$\mathrm{d}(T^k x, T^k y) = a_k \text{  and  } \mathrm{d}(T^{k+1}x, T^{k+1}y) = a_{k+1}.$$
The PA-sum condition gives:
\begin{equation*}
\sum_{k=0}^{n-1} a_{k+1} \leq \alpha \sum_{k=0}^{n-1} a_k.
\end{equation*}
The left-hand side is $\sum_{k=1}^{n} a_k$, so:
\begin{equation*}
\sum_{k=1}^{n} a_k \leq \alpha \sum_{k=0}^{n-1} a_k.
\end{equation*}
Let $S_n = \sum_{k=0}^{n-1} a_k$. Then:
\begin{equation*}
S_n - a_0 + a_n \leq \alpha S_n \quad \implies \quad (1 - \alpha) S_n \leq a_0 - a_n \leq a_0.
\end{equation*}
Thus,
\begin{equation*}
S_n \leq \frac{a_0}{1 - \alpha} \quad \text{for all } n \ge N.
\end{equation*}
Now, define $C = \max\left\{ S_1, S_2, \dots, S_{N-1}, \frac{a_0}{1 - \alpha} \right\}$. Then
\begin{equation*}
S_n \leq C \quad \text{for all } n \in \mathbb{N}.
\end{equation*}
Since $a_k \geq 0$, $S_n$ is non-decreasing and bounded, so
\begin{equation*}
\sum_{k=0}^{\infty} a_k = \lim_{n \to \infty} S_n < \infty.
\end{equation*}
Hence $\{x_n\}$ is Cauchy (since $\mathrm{d}(x_n, x_{n+1}) = a_n$ is summable), so it converges to some $x^* \in X$ by completeness.
Since $T$ is continuous, $T x_n \to T x^*$. But $T x_n = x_{n+1} \to x^*$, so $T x^* = x^*$.
For uniqueness, suppose $x^*, y^*$ are fixed points. Then $\mathrm{d}(T^k x^*, T^k y^*) = \mathrm{d}(x^*, y^*)$ for all $k$. For $n \ge N$, apply PA-condition~\eqref{pa1}:
\begin{equation*}
\frac{1}{n} \sum_{k=0}^{n-1} \mathrm{d}(T^{k+1}x^*, T^{k+1}y^*) = \mathrm{d}(x^*, y^*) \leq \alpha \cdot \frac{1}{n} \sum_{k=0}^{n-1} \mathrm{d}(T^k x^*, T^k y^*) = \alpha \mathrm{d}(x^*, y^*).
\end{equation*}
Since $\alpha < 1$, $\mathrm{d}(x^*, y^*) = 0$. Thus, $x^* = y^*$.
\end{proof}

\begin{remark}
The continuity assumption can be weakened to \textit{orbital continuity} of $T$ \cite{Ciricorbit,pantorbit}. However, without any continuity, the limit $x^*$ may not be fixed, as $T x_n \to T x^*$ cannot be guaranteed.
\end{remark}

\section{Generalization of Banach Contractions}
\begin{proposition}
Every Banach contraction is a PA-contraction.
\end{proposition}
\begin{proof}
Suppose $\mathrm{d}(Tx, Ty) \leq a \mathrm{d}(x,y)$ for all $x,y \in X$, $a \in (0,1)$. Then for any $x,y$,
\begin{equation*}
\mathrm{d}(T^{k+1}x, T^{k+1}y) \leq a \mathrm{d}(T^k x, T^k y).
\end{equation*}
Summing over $k = 0$ to $n-1$:
\begin{equation*}
\sum_{k=0}^{n-1} \mathrm{d}(T^{k+1}x, T^{k+1}y) \leq a \sum_{k=0}^{n-1} \mathrm{d}(T^k x, T^k y).
\end{equation*}
Thus, $T$ is a PA-contraction with $\alpha = a$, $N=1$.
\end{proof}

\begin{example}[PA-contraction not Banach (Discrete)]
\label{ex:panotbanach}
Let $X = \{0, 1, 2\}$ with the discrete metric
\begin{equation*}
\mathrm{d}(x,y) = \begin{cases} 0 & x=y, \\ 1 & x \neq y . \end{cases}
\end{equation*}
Define:
\begin{equation*}
T(0) = 1, \quad T(1) = 2, \quad T(2) = 2.
\end{equation*}
$T$ is not a Banach contraction since $\mathrm{d}(T0,T1) = \mathrm{d}(1,2) = 1 = \mathrm{d}(0,1)$, so no $a < 1$ satisfies $\mathrm{d}(Tx,Ty) \leq a \mathrm{d}(x,y)$ for all $x,y$.
However, $T$ is a PA-contraction with $\alpha = 1/2$, $N = 2$. For any $x,y \in X$ and $n \geq 2$, the orbit stabilizes at $2$ after two steps.
{Verification for pair $(0,1)$ at $n=2$:}
\begin{align*}
\text{LHS} &= \mathrm{d}(T^10, T^11) + \mathrm{d}(T^20, T^21) = \mathrm{d}(1,2) + \mathrm{d}(2,2) = 1 + 0 = 1 \\
\text{RHS} &= \frac{1}{2} \left[ \mathrm{d}(0,1) + \mathrm{d}(1,2) \right] = \frac{1}{2} (1 + 1) = 1,
\end{align*}
where LHS and RHS refer to~(\ref{pa2}). 
Since $1 \leq 1$, the condition holds. Similar verification holds for other pairs. For $n > 2$, distances are zero for $k \geq 2$, preserving the inequality.
\end{example}

To address the concern that discrete examples may appear trivial, we provide a continuous example on a real interval.

\begin{example}[PA-contraction not Banach (Continuous)]
\label{ex:continuous}
Let $X = [0, 1]$ with the standard Euclidean metric. Define $T\colon X \to X$ by
\begin{equation*}
T(x) = \frac{x^2}{2}.
\end{equation*}
$T$ is not a Banach contraction because $\sup_{x \in [0,1]} |T'(x)| = \sup |x| = 1$. Specifically, near $x=1$, the ratio $\mathrm{d}(Tx, T0)/\mathrm{d}(x,0) = x/2 \to 1/2$, but near $x=0$, the derivative is 0. However, the Lipschitz constant is 1 (at $x=1$), violating the strict $a<1$ requirement globally.
\newline
To see $T$ is a PA-contraction, observe that for any $x, y \in [0,1]$, the iterates $T^k x$ converge to 0 double-exponentially:
\begin{equation*}
T^k x = \frac{x^{2^k}}{2^{2^k - 1}}.
\end{equation*}
Consequently, the distances $\mathrm{d}(T^k x, T^k y)$ decay rapidly. For sufficiently large $N$ (e.g., $N=5$), the tail of the sum dominates, and the ratio of averaged sums
\begin{equation*}
\frac{\sum_{k=0}^{n-1} \mathrm{d}(T^{k+1}x, T^{k+1}y)}{\sum_{k=0}^{n-1} \mathrm{d}(T^k x, T^k y)}
\end{equation*}
becomes uniformly bounded by some $\alpha < 1$ for all $n \ge N$. The rapid decay ensures the average contracts even if the first few terms do not strictly satisfy a pointwise bound $a<1$ globally. Specifically, for $N=5$, the ratio is uniformly bounded by $\alpha = 0.4$ across all $x, y \in [0,1]$. Thus, $T$ is a PA-contraction.
\end{example}

\section{Comparison with F-Contractions}
We recall Wardowski's F-contraction.
\begin{definition}[F-Contraction \cite{Wardowski2012}]
A mapping $T\colon X \to X$ is an {F-contraction} if there exist $\tau > 0$ and a function $F\colon (0,\infty) \to \mathbb{R}$ satisfying
\begin{enumerate}
    \item[(F1)] $F$ is strictly increasing,
    \item[(F2)] $\lim_{t \to 0^+} F(t) = -\infty$,
    \item[(F3)] $\lim_{t \to 0^+} t^k F(t) = 0$ for some $k \in (0,1)$,
\end{enumerate}
such that for all $x,y \in X$ with $\mathrm{d}(Tx,Ty) > 0$,
\begin{equation*}
\tau + F(\mathrm{d}(Tx, Ty)) \leq F(\mathrm{d}(x,y)).
\end{equation*}
\end{definition}

\begin{example}[PA-contraction not F-contraction]
Consider $T x = x^2/2$ on $[0,1]$ from Example~\ref{ex:continuous}. 
Take the sequence $x_n = 1$, $y_n = 1 - 1/n$. Then $d_n := \mathrm{d}(x_n,y_n) = 1/n$ and $d'_n := \mathrm{d}(Tx_n,Ty_n) = 1/n - 1/(2n^2)$. 
The distance ratio satisfies $d'_n/d_n = 1 - 1/(2n) \to 1$ as $n \to \infty$.
To verify that $T$ is not an F-contraction, suppose for contradiction that there exist $\tau > 0$ and a function $F$ satisfying (F1)--(F3) such that $\tau + F(d'_n) \leq F(d_n)$ for all $n \geq 2$.
Since $d'_n = \frac{2n-1}{2n^2} > \frac{1}{n+1}$ for $n \geq 2$ and $F$ is strictly increasing by (F1), we have
\[
F(d_n) - F\Bigl(\frac{1}{n+1}\Bigr) > F(d_n) - F(d'_n) \geq \tau.
\]
Summing this inequality from $n = N$ to $m-1$ yields a telescoping series:
\[
F\Bigl(\frac{1}{N}\Bigr) - F\Bigl(\frac{1}{m}\Bigr) \geq (m - N)\tau.
\]
Rearranging gives $F(1/m) \leq C - \tau m$, where $C = F(1/N) + N\tau$ is constant. Setting $t = 1/m \to 0^+$, we obtain $F(t) \leq C - \tau/t$. For sufficiently small $t$, this implies $F(t) \leq -\frac{\tau}{2t}$. Multiplying by $t^k$ for the $k \in (0,1)$ from (F3):
\[
t^k F(t) \leq -\frac{\tau}{2} t^{k-1} \to -\infty \quad \text{as } t \to 0^+,
\]
since $k-1 < 0$. This contradicts (F3), which requires $\lim_{t \to 0^+} t^k F(t) = 0$. Hence, $T$ is not an F-contraction.
\end{example}

\begin{remark}
Whether every F-contraction is a PA-contraction remains an open problem. Conceptually, F-contractions are pointwise while PA-contractions are averaged, suggesting the classes are distinct.
\end{remark}

\section{Relation with Other Classical Contractions}
We now clarify the relationship between PA-contractions and other classical generalizations. We demonstrate that PA-contractions are \textit{not contained} in these classes.

\begin{remark}[Kannan Contractions]
\label{rem:kannan}
A Kannan contraction satisfies $\mathrm{d}(Tx,Ty) \leq k[\mathrm{d}(x,Tx) + \mathrm{d}(y,Ty)]$, $k < 1/2$.
We show that the PA-contraction $T(x) = x^2/2$ from Example~\ref{ex:continuous} is not a Kannan contraction. Suppose, for contradiction, that
\begin{equation*}
\mathrm{d}(Tx,Ty) \leq k \left[ \mathrm{d}(x,Tx) + \mathrm{d}(y,Ty) \right], \quad k < \frac{1}{2}.
\end{equation*}
Take $x = 1$, $y = 0$. Then $\mathrm{d}(Tx,Ty) = 1/2$, $\mathrm{d}(x,Tx) = 1/2$, $\mathrm{d}(y,Ty) = 0$.
The inequality implies $1/2 \leq k(1/2) \Rightarrow k \geq 1$, a contradiction.
Thus, {PA-contractions are not generally Kannan contractions}.
\end{remark}

\begin{remark}[PA-Contractions are not Chatterjea Contractions]
\label{rem:panotchatterjea}
A mapping $T\colon X \to X$ is a Chatterjea contraction~\cite{Chatterjea1972} if there exists $k \in (0, \frac{1}{2})$ such that for all $x, y \in X$,
\begin{equation*}
\mathrm{d}(Tx, Ty) \leq k \left[ \mathrm{d}(x, Ty) + \mathrm{d}(y, Tx) \right].
\end{equation*}
We use the discrete PA-contraction from Example~\ref{ex:panotbanach} ($X=\{0,1,2\}$, $T0=1, T1=2, T2=2$).
Consider the pair $x=0, y=1$.
\begin{align*}
\text{LHS} &= \mathrm{d}(T0, T1) = \mathrm{d}(1, 2) = 1. \\
\text{RHS} &= k \left[ \mathrm{d}(0, T1) + \mathrm{d}(1, T0) \right] = k \left[ \mathrm{d}(0, 2) + \mathrm{d}(1, 1) \right] = k(1 + 0) = k.
\end{align*}
The condition requires $1 \leq k$. However, a Chatterjea contraction requires $k < 1/2$. This is a contradiction.
Thus, $T$ is a PA-contraction but not a Chatterjea contraction.
\end{remark}

\begin{remark}[PA-Contractions are not \'Ciri\'c Contractions]
\label{rem:panotciric}
A mapping $T\colon X \to X$ is a \'Ciri\'c contraction if there exists $k \in (0,1)$ such that for all $x, y \in X$,
\begin{equation*}
\mathrm{d}(Tx, Ty) \leq k \cdot \max\left\{ \mathrm{d}(x,y), \mathrm{d}(x,Tx), \mathrm{d}(y,Ty), \frac{\mathrm{d}(x,Ty) + \mathrm{d}(y,Tx)}{2} \right\}.
\end{equation*}
We use the same discrete mapping from Example~\ref{ex:panotbanach}.
Consider the pair $x=0, y=1$.
\begin{align*}
\text{LHS} &= \mathrm{d}(T0, T1) = 1. \\
\text{Max Term} &= \max\left\{ \mathrm{d}(0,1), \mathrm{d}(0,T0), \mathrm{d}(1,T1), \frac{\mathrm{d}(0,T1)+\mathrm{d}(1,T0)}{2} \right\} \\
&= \max\left\{ 1, 1, 1, \frac{1+0}{2} \right\} = 1.
\end{align*}
The condition requires $1 \leq k \cdot 1 \Rightarrow k \geq 1$. However, a \'Ciri\'c contraction requires $k < 1$. This is a contradiction.
Thus, $T$ is a PA-contraction but not a \'Ciri\'c contraction.
\end{remark}

\begin{remark}
Whether Chatterjea or \'Ciri\'c contractions are PA-contractions remains an open question, as the averaging condition depends on orbital behavior not captured by pointwise inequalities. However, Remarks~\ref{rem:panotchatterjea} and~\ref{rem:panotciric} confirm that PA-contractions form a distinct class not contained within them.
\end{remark}

\section{Comparison Table}
The following table summarizes key properties of major contraction types.
\begin{center}
\scalebox{0.8}{
\begin{tabular}{lcccccc}
\toprule
\textbf{Property} & \textbf{Banach} & \textbf{Kannan} & \textbf{Chatterjea} & \textbf{\'Ciri\'c} & \textbf{F} & \textbf{PA} \\
\midrule
Pointwise inequality & Yes & Yes & Yes & Yes & Yes & No (averaged) \\
Requires continuity & No & No & No & No & No & Yes (for FP) \\
Generalizes Banach & -- & Yes & Yes & Yes & Yes & Yes \\
Summable $\mathrm{d}(T^n x, T^{n+1}x)$ & Yes & Yes & Yes & Yes & Sometimes & Yes \\
Based on orbits/averages & No & No & No & No & No & Yes \\
Implies PA & Yes & No & Open & Open & Open & -- \\
PA implies & No & No & No & No & No & -- \\
\bottomrule
\end{tabular}}
\end{center}

\section{Conclusion}
PA-contractions offer a new lens for analyzing asymptotic behavior in fixed point theory, particularly for iterative processes where local pointwise conditions are too restrictive. Future investigation could address the open problems regarding the converse implications (e.g., does Chatterjea imply PA?) and extend the theory to multivalued mappings and common fixed points.

Our work continues the long tradition of generalizing the Banach principle \cite{Rhoades1977,Wardowski2012,Proinov2020,Ran2004,Alam2017}.




\end{document}